\documentclass[12pt]{amsart}

 \usepackage{amsfonts,graphics,amsmath,amsthm,amsfonts,amscd,amssymb,amsmath,latexsym,multicol,
 mathrsfs}
\usepackage{epsfig,url}
\usepackage{flafter}
\usepackage{fancyhdr}
\usepackage{hyperref}
\hypersetup{colorlinks=true, linkcolor=black}

\makeatletter

\def\jobis#1{FF\fi
  \def\predicate{#1}%
  \edef\predicate{\expandafter\strip@prefix\meaning\predicate}%
  \edef\job{\jobname}%
  \ifx\job\predicate
}

\makeatother

\if\jobis{proposal}%
\else
\fi

 \usepackage[matrix, arrow]{xy}

\DeclareMathOperator{\Supp}{Supp}

 \numberwithin{equation}{subsection}
 \numberwithin{footnote}{subsection}

 \newtheorem{prop}[subsection]{Proposition}
 \newtheorem{thm}[subsection]{Theorem}
 \newtheorem{conj}[subsection]{Conjecture}

{
    \newtheoremstyle{upright}%
        {8pt plus2pt minus4pt}%
        {8pt plus2pt minus4pt}%
        {\upshape}%
        {}%
        {\bfseries\scshape}%
        {}%
        {1em}%
        {}%
\theoremstyle{upright}

 \newtheorem{exa}[subsection]{Example}

 \newtheorem{rem}[subsection]{Remark}

}

 \newcommand{\x}{\mathscr}
 \newcommand{\C}{\mathbb C}
 
 \newcommand{\PP}{\mathbb P}
 
 \newcommand{\Q}{\mathbb Q}
 \newcommand{\R}{\mathbb R}
 \newcommand{\Z}{\mathbb Z}
 \newcommand{\bir}{\dashrightarrow}
 \newcommand{\rddown}[1]{\left\lfloor{#1}\right\rfloor} 

\title{On the moduli part of the Kawamata-Kodaira canonical bundle formula}
\author{Caucher Birkar \hspace{2cm} Yifei Chen}
\date{\today}
\begin{document}
\maketitle

\begin{abstract}
It is conjectured that the moduli b-divisor of 
the Kawamata-Kodaira canonical bundle formula associated to a klt-trivial fibration 
$(X,B)\to Z$ is semi-ample. 
In this paper, we show the semi-ampleness of an arbitrarily small perturbation of 
the moduli b-divisor by a fixed appropriate divisor which roughly speaking comes from a section of 
$K_X+B$.
 
We apply the above result to settle a conjecture of Fujino 
and Gongyo: if  $f\colon X\to Z$ is a smooth surjective morphism of smooth 
projective varieties with $-K_X$ semi-ample, then $-K_Z$ is also semi-ample.  
We list several counter-examples to show that this fails without the smoothness assumption on $f$.
\end{abstract}



\section{Introduction}

\textbf{The semi-ampleness conjecture.}
We work over $\C$.
Let $f\colon X\to Z$ be a contraction of normal projective varieties, and
$(X,B)$ klt such that $K_X+B\sim_\Q 0/Z$, i.e. $K_X+B\sim_\Q f^*N$ for
some $\Q$-Cartier $\Q$-divisor $N$. Such a contraction is called a \emph{klt-trivial fibration}.
By a construction of Kawamata \cite{ka97},\cite{ka98} we have
a decomposition
$$
N\sim_\Q K_Z+B_Z+M_Z
$$
where
$B_Z$ is defined using the singularities of $(X,B)$ and of the fibres of $f$ over the codimension one 
points of $Z$, and $(Z,B_Z)$ is klt
if $K_Z+B_Z$ is $\Q$-Cartier. The part $B_Z$ is called the \emph{discriminant part} and the
part $M_Z$ is called the \emph{moduli part}. More precisely, $B_Z$ is defined as follows: 
for each prime divisor $Q$ on $Z$, let $t$ be the lc threshold of $f^*Q$ over the 
generic point of $Q$, with respect 
to the pair $(X,B)$; then let $(1-t)$ be the coefficient of $Q$ in $B_Z$. Except for 
finitely many $Q$, $t=1$ hence $B_Z$ has finitely many components.

Consider a commutative diagram
$$
\xymatrix{
X' \ar[r]^{\tau} \ar[d]^{f'} &  X\ar[d]^{f}\\
Y \ar[r]^{\sigma} & Z }
$$
in which $X',Y$ are normal and projective, $\sigma,\tau$ are birational, and $f'$ is a contraction. 
Let $K_{X'}+B':=\tau^*(K_X+B)$ where $B'$ might have negative coefficients. 
Using the relation $K_{X'}+B'\sim_\Q 0/Y$, we can similarly define a decomposition 
$$
\sigma^*N\sim_\Q K_{Y}+B_{Y}+M_{Y}
$$
where:  $B_Z=\sigma_*B_{Y}$ and $M_Z=\sigma_*M_{Y}$. This means that the moduli part is a b-divisor. 
In general, $B_Y$ may have negative coefficients.

Kawamata \cite{ka98} showed that $M_{Y}$ is nef if $Y$ is a sufficiently high resolution.
Ambro \cite{am04} proved that the moduli part is a b-Cartier divisor, i.e.
we can find a resolution $Y\to Z$  so that
for any other resolution $Y'\to Y$ the moduli part $M_{Y'}$ is just the pullback of $M_{Y}$.
Moreover,  Ambro \cite{am05} proved that $M_Y$ is actually
the pullback of some nef and big divisor for some contraction $g\colon Y\to T$.
We call $Y$ an \emph{Ambro model.}

\begin{conj}[cf. \cite{ka97}, \cite{am99},{\cite{psh09}}]
If $Y$ is an Ambro model, then $M_{Y}$ is semi-ample.
\end{conj}

This seems to be a difficult conjecture. Only some very special cases are 
known:  when $M_Y\equiv 0$ \cite{am05}, when  $\dim X=\dim Z+1$ \cite{ka97}\cite{psh09}, 
and when the geometric generic fibre of $f$ is birationally an abelian variety \cite{f03}.\\

\textbf{Results of this paper.}
We are not able to prove the semi-ampleness of $M_Y$ but we prove semi-ampleness of 
a "nearby" divisor. This would be enough for some applications. 
To state our main result concerning $M_Y$ we need to introduce one more notation. 
Since $(Y,B_Y)$ is log smooth and each coefficient of $B_Y$ is strictly
smaller than $1$, there is an exceptional$/Z$ $\Q$-divisor
$E\ge 0$ on $Y$ such that $(Y,\Delta_{Y})$ is klt where $\Delta_{Y}:=B_{Y}+E\ge 0$.

\begin{thm}\label{t-main}
Under the above notation and assumptions, suppose that  we have $\kappa(K_Y+\Delta_Y/T)\ge 0$, 
i.e. suppose that there is a rational number $b\ge 0$
such that
$$
K_Y+\Delta_Y+bM_Y\sim_\Q D_Y ~~\mbox{for some $D_Y\ge 0$}
$$
Then, we can choose $D_Y\ge 0$ in its $\Q$-linear equivalence class
so that there exist a resolution $\phi\colon W\to Y$ and
a $\Q$-divisor $G$ on $W$ such that

$\bullet$ $0\le G\le \phi^*D_Y$, 

$\bullet$ $\alpha M_{W}+G=\alpha \phi^*M_Y+G$ is semi-ample for any $\alpha\gg 0$, and

$\bullet$ $R(W,lG)\simeq R(Y,lD_Y)$ for some integer $l>0$.
\end{thm}

Here $R(-)$ stands for the section ring (see Section 2 for precise definition). 
The proof is based on the minimal model program.
An interesting special case of the theorem is when $K_X+B\sim_\Q 0$ in which case
we take $D_Y=E$.
Even this special case is strong enough to imply the following result which
was conjectured by Fujino and Gongyo \cite{fg11} (see also \cite{fg12}).

\begin{thm}\label{t-FG-conj}
Let $f\colon X\to Z$ be a smooth surjective morphism of smooth projective varieties.
If $-K_X$ is semi-ample, then $-K_Z$ is also semi-ample.
\end{thm}

Fujino-Gongyo proved the conjecture when $-K_X$ is big. They also proved some other special cases 
(see Remark 4.2 and Theorem 4.4 of \cite{fg11}).
The theorem does not hold  without the smoothness assumption on $f$ by a counter-example due to H. Sato 
(see Section 4 where we give additional counter-examples).\\

 \textbf{Relevant results.}
There are results similar to Theorem \ref{t-FG-conj} in the literature. We mention some 
of them for completeness. Let 
$f\colon X\to Z$ be a smooth surjective morphism of smooth projective varieties. 
Then, it is known that (see \cite{fg11}, \S 5 for more details):\\

$\bullet$ if $-K_X$ is nef, then $-K_Z$ is nef (Miyaoka \cite{m93})

$\bullet$ if $-K_X$ is ample, then $-K_Z$ is ample (Koll\'ar-Miyaoka-Mori \cite{kmm92})

$\bullet$ if $-K_X$ is nef and big, then $-K_Z$ is nef and big (Fujino-Gongyo \cite{fg11}). \\

To prove Theorem \ref{t-FG-conj}, we do not rely on these known results but we use ideas 
of \cite{fg11}.

In general,  varieties with nef anticanonical divisor have interesting properties. 
Let $X$ be a smooth projective variety with $-K_X$ nef. Then, for any surjective morphism 
$g\colon X\to Y$ with $Y$ smooth projective,  Zhang \cite{zh96} proved that either $Y$ is uniruled or $\kappa(Y)=0$.
On the other hand,  Lu, Tu, Zhang, and Zheng \cite{ltzz10} 
proved that the Albanese map $X\to A$ is semi-stable (it is conjectured that this map should be smooth).\\

{\textbf{Acknowledgements.} We would like to thank Osamu Fujino and Florin Ambro for their 
comments and answering our questions. This work started when Caucher was 
visiting Yifei at the Mathematics Institute of the Chinese Academy of Sciences
in Beijing; he would like to thank them for their hospitality. 
This work was completed when Yifei was visiting Caucher in Cambridge and he
would like to thank Caucher and his family for the hospitality.
Both authors were supported by a Leverhulme grant.

\section{Preliminaries}

\textbf{Pairs.} We work over $k=\C$. 
A \emph{pair} $(X,B)$ consists of a normal quasi-projective variety $X$ over $k$ and 
a $\Q$-divisor $B$ on $X$ with
coefficients in $[0,1]$ (called a boundary) such that $K_X+B$ is $\mathbb{Q}$-Cartier. 
A pair $(X,B)$ is called \emph{Kawamata log terminal} (klt) if for any 
projective birational morphism $g\colon Y\to X$ from a normal variety $Y$, every coefficient 
of $B_Y$ is less than one where $K_Y+B_Y:=g^*(K_X+B)$. 
For basic properties of singularities and other aspects of birational geometry 
we refer the reader to \cite{km98}.

\textbf{Minimal models.}
Let $X$ be a normal projective variety and $D$ a $\Q$-Cartier $\Q$-divisor on $X$. 
A normal projective variety $Y$  
is called a \emph{minimal model} of $D$ 
if there is a birational map $\phi\colon X\bir Y$ 
such that $\phi^{-1}$ does not contract divisors, $Y$ is $\Q$-factorial,
${D}_Y:=\phi_*D$ is nef, 
and there is a common resolution $g\colon W\to X$ and $h\colon W\to Y$ such that 
$E:=g^*D-h^*D_Y$ is effective and $\Supp g_*E$ is equal to the union of the 
exceptional divisors of $\phi$. 
If ${D}_Y$ is semi-ample, we call ${Y}$ 
a \emph{good} minimal model of $D$. 
A standard way of obtaining a minimal model of $D$ is to run a \emph{minimal 
model program} on $D$, if possible. Such a program can be defined exactly 
as in the case $D=K_X+B$ for a klt pair $(X,B)$. 

When $(X,B)$ is klt and $D=K_X+B$ the above definition of minimal model  
is equivalent to the usual definition of log minimal models. 
Moreover, if $K_X+B$ is pseudo-effective 
and $B$ is big, then $K_X+B$ has a good log minimal model by \cite{bchm}.

\textbf{Section rings.}
For a $\Q$-divisor $D$ on a
normal projective variety $X$, the \emph{section ring} of $D$ is defined as
$$
R(X,D)=\bigoplus_{0\le m\in \Z} H^0(X,\rddown{mD})
$$
If  $\phi\colon X\bir Y$ is a partial minimal model program on $D$, i.e. a finite 
sequence of divisorial contractions and flips with respect to $D$,
then there is a common resolution $g\colon W\to X$ and $h\colon W\to Y$ such that 
$E:=g^*D-h^*{D}_Y$ is effective and $\Supp g_*E$ is equal to the union of the 
exceptional divisors of $\phi$. In particular, this implies that 
 there is a natural isomorphism
$$
R(X,D)\simeq R(Y,D_Y)
$$

\textbf{Kodaira dimension.}
Let $f\colon X\to Z$ be a contraction of normal varieties and $D$ a $\Q$-Cartier $\Q$-divisor 
 on $X$. By $\kappa(D/Z)$ we mean the Kodaira dimension of $D|_F$ where 
$F$ is the generic fibre of $f$.\\

\section{Proof of Theorem \ref{t-main}}

In this section, we prove the following result which is of independent interest and it
immediately implies Theorem \ref{t-main}. The proof also works for $\R$-divisors but 
for simplicity we only consider $\Q$-divisors.

\begin{prop}\label{p-aux}
Let $(Y,\Delta_Y)$ be a projective klt pair with $\Delta_Y$ a $\Q$-divisor, $g\colon Y\to T$ a contraction,
$M_T$ a nef and big $\Q$-divisor on $T$, and $M_Y=g^*M_T$. Assume that $b\ge 0$ is a rational number
such that
$$
K_Y+\Delta_Y+bM_Y\sim_\Q D_Y ~~\mbox{for some $D_Y\ge 0$}
$$
 Then, we can choose $D_Y\ge 0$ in its $\Q$-linear equivalence class such that
there exist a resolution $\phi\colon W\to Y$ and
a $\Q$-divisor $G$ on $W$ such that

$\bullet$ $0\le G\le \phi^*D_Y$, 

$\bullet$ $\alpha \phi^*M_{Y}+G$ is semi-ample for any $\alpha\gg 0$, and 

$\bullet$ $R(W,lG)\simeq R(Y,lD_Y)$ for some integer $l>0$.
\end{prop}
\begin{proof}
For any morphism $U\to T$, $M_U$ will denote the pullback of $M_T$.

\emph{Step 1.} Let $\pi\colon Y'\to Y$ be any log resolution of $(Y,\Delta_Y)$. There is
an exceptional$/Y$ $\Q$-divisor
$E'\ge 0$ and a boundary $\Delta_{Y'}$ such that $(Y',\Delta_{Y'})$ is klt and
$$
K_{Y'}+\Delta_{Y'}=\pi^*(K_Y+\Delta_Y)+E'
$$
Let $D_{Y'}=\pi^* D_Y+E'$. Then, 
$$
K_{Y'}+\Delta_{Y'}+bM_{Y'}\sim_\Q D_{Y'}
$$
Now assume that the result holds on $Y'$, that is, assume that
we can choose $D_{Y'}\ge 0$ in its $\Q$-linear equivalence class such that
there exist a resolution $\psi\colon W\to Y'$ and
a $\Q$-divisor $G$ on $W$ satisfying:

$\bullet$ $0\le G\le \psi^*D_{Y'}$, 

$\bullet$ $\alpha M_{W}+G$ is semi-ample for any $\alpha\gg 0$, and 

$\bullet$ $R(W,lG)\simeq R(Y',lD_{Y'})$ for some integer $l>0$.\\

We will show that the result also holds for $Y$ by taking $\phi=\pi\psi$.
Since $D_{Y'}=\pi^* D_Y+E'$ and $E'\ge 0$ is exceptional/$Y$,
$$
R(Y,lD_{Y})\simeq R(Y',lD_{Y'})\simeq R(W,lG)
$$ 
It only remains to show that $G\le \phi^*D_{Y}$.
Put $C=\phi^*D_{Y}-G$. Since $E'$ is exceptional$/Y$,
$$
\phi_*C=\phi_*(\phi^*D_{Y}-G)=\phi_*(\phi^*D_{Y}+\psi^*E'-G)=\phi_*(\psi^*D_{Y'}-G)\ge 0
$$
On the other hand, $\alpha M_{W}+G$ is semi-ample for any $\alpha\gg 0$ and
$M_W\sim_\Q 0/Y$ hence $G$ is semi-ample over $Y$. So, $C$ is antinef over $Y$
which implies that $C\ge 0$ by the negativity lemma. Therefore, $G\le \phi^*D_{Y}$.\\

\emph{Step 2.}
Since  $M_Y\sim_\Q 0/T$,  $K_Y+\Delta_Y\sim_\Q D_Y/T$ hence  $\kappa(K_Y+\Delta_Y/T)\ge 0$.
By taking a log resolution and applying Step 1 we could assume that the relative Iitaka fibration of
$K_Y+\Delta_Y$ over $T$ is a morphism $Y\to S'/T$ where $S'$ is smooth.
In particular, this means that $\kappa(K_Y+\Delta_Y/S')=0$.
By applying Fujino-Mori \cite{fm00}, we can find a commutative diagram
$$
\xymatrix{
  & V \ar[ld]^\tau \ar[rd]^h & \\
 Y \ar[rd]^g & & S\ar[ld]^\theta \\
  & T &
}
$$
where $\tau$ is a resolution, and $h,\theta$ are contractions
such that we have the following data:

$\bullet$ a klt $(S,\Delta_S)$ where  $\Delta_S$ is a $\Q$-divisor,

$\bullet$ a nef $\Q$-divisor $L_S$ on $S$,

$\bullet$ $K_S+\Delta_S+L_S$ is big$/T$,

$\bullet$ $\kappa(K_{Y}+\Delta_{Y}/T)=\kappa(K_S+\Delta_S+L_S/T)=\dim S-\dim T$,

$\bullet$ and
$$
\tau^*(K_{Y}+\Delta_{Y})+R^-\sim_\Q h^*(K_S+\Delta_S+L_S)+R^+
$$
where $R^+,R^-$
are effective $\Q$-divisors with $R^-$ exceptional$/Y$, $h(R^-)$ has no
codimension one components, and
$h_*\x{O}_V(\rddown{iR^+})=\x{O}_S$ for every $i>0$.

 The above properties imply that for any rational number $\alpha\ge 0$, we have
 
\begin{equation*}
\begin{split}
 \tau^*D_Y+R^-+\alpha M_V & \sim_\Q \tau^*(K_{Y}+\Delta_{Y}+bM_Y)+R^-+\alpha M_V\\
&= \tau^*(K_{Y}+\Delta_{Y})+R^-+bM_V+\alpha M_V\\
& \sim_\Q h^*(K_S+\Delta_S+L_S)+R^++bM_V+\alpha M_V\\
& = h^*(K_S+\Delta_S+L_S+bM_S+\alpha M_S)+R^+\\
\end{split}
\end{equation*}

\emph{Step 3.}
Fix a rational number $\lambda\ge 0$ so that $b+\lambda>0$. 
By construction
$$
0\le \kappa(D_Y)=\kappa(K_{Y}+\Delta_{Y}+bM_{Y})=\kappa(K_S+\Delta_S+L_S+bM_S)
$$
so there is a $\Q$-divisor $D_S\ge 0$ such that
$$
D_S\sim_\Q K_S+\Delta_S+L_S+bM_S
$$
Since $D_S$ is big$/T$ and
$M_T$ is big, there is $a>0$ such that $D_S+aM_S$ is big. Thus, we have
$$
D_S+aM_S\sim_\Q A+B
$$
where $A$ is ample and $B\ge 0$.

 Pick $\epsilon>0$ sufficiently small. This ensures that 
$$
\beta_\lambda:=\epsilon \lambda-\epsilon a+b+\lambda\ge 0
$$
Moreover, since $A$ is ample and $L_S$ and $M_S$ are nef,
there is a big boundary $\Gamma_{S,\lambda}$ such that $(S,\Gamma_{S,\lambda})$ is klt and such that
 we can write
\begin{equation*}
\begin{split}
K_S+\Gamma_{S,\lambda} &  \sim_\Q K_S+\Delta_S+L_S+\epsilon A+\epsilon B +\beta_\lambda M_S\\
& \sim_\Q  K_S+\Delta_S+L_S+\epsilon D_S+\epsilon aM_S+\beta_\lambda M_S\\
& \sim_\Q  K_S+\Delta_S+L_S+b M_S+\epsilon D_S+\epsilon \lambda M_S+\lambda M_S \\
& \sim_\Q  (1+\epsilon)(K_S+\Delta_S+L_S+b M_S) + (1+\epsilon)\lambda M_S\\
& \sim_\Q (1+\epsilon)(K_S+\Delta_S+L_S+b M_S+\lambda M_S)\\
& \sim_\Q (1+\epsilon)(D_S+\lambda M_S)
\end{split}
\end{equation*}
By \cite{bchm}, we can run an LMMP on $K_S+\Gamma_{S,\lambda}$
which ends up with a good log minimal model of $K_S+\Gamma_{S,\lambda}$. By the relations above,
this also produces a good minimal model of $K_S+\Delta_S+L_S+b M_S+\lambda M_S$.

For any $\alpha\ge \lambda$, we have 
$$
K_S+\Gamma_{S,\alpha}\sim_\Q K_S+\Gamma_{S,\lambda}+ (1+\epsilon)(\alpha-\lambda)M_S
$$
If $\alpha\gg \lambda$ and if we run an LMMP on $K_S+\Gamma_{S,\alpha}$, then 
$M_S$ is numerically trivial on each extremal ray contracted in the process: 
this follows from the boundedness of the length of extremal rays due to Kawamata \cite{ka91}; 
indeed if $R$ is an extremal ray such that $(K_S+\Gamma_{S,\alpha})\cdot R<0$, then 
$(K_S+\Gamma_{S,\lambda})\cdot R<0$ and there is a rational curve $C$ generating $R$  
such that 
$$
-2\dim S\le (K_S+\Gamma_{S,\lambda})\cdot C<0
$$ 
If $\alpha\gg \lambda$  
and if $M_S\cdot C>0$, then  $(K_S+\Gamma_{S,\alpha})\cdot C>0$, a contradiction.
The same argument applies in each step of the LMMP because the Cartier index of $M_S$ is 
preserved by the LMMP (by Cartier index of M$_S$ we mean the smallest natural number $n$ 
such that $nM_S$ is Cartier). Therefore,  
the LMMP on $K_S+\Gamma_{S,\alpha}$ is also an LMMP on $K_S+\Gamma_{S,\alpha'}$ 
for any $\alpha'\ge \alpha$.
In particular, if $\overline{S}$ is the good minimal model of $K_S+\Gamma_{S,\alpha}$ 
obtained by the LMMP above, then 
$\overline{S}$ is also a good minimal model of $K_S+\Gamma_{S,\alpha'}$ for any $\alpha'\ge \alpha$.\\

\emph{Step 4.}
Let $\overline{S}$ be the model constructed in Step 3.
 By construction, there is a
commutative diagram
$$
\xymatrix{
&& W \ar[ld]^\mu \ar[rd]^e &&\\
  & V \ar[ld]^\tau \ar[rd]^h & & \overline{W} \ar[ld]^c \ar[rd]^{d} &\\
 Y \ar[rd]^g & & S \ar[ld]^\theta \ar@{-->}[rr] & & \overline{S}\\
  & T &  &
}
$$
where $\mu$ is a resolution, $e$ is a contraction, and $c,d$ are also resolutions.
By Step 3, we have $e^*d^*M_{\overline{S}}=e^*c^*M_S=M_W$ because $M_S$ is numerically trivial on each step of 
the LMMP that produced $S\bir \overline{S}$. Moreover, 
$D_{\overline{S}}+\alpha M_{\overline{S}}$ is semi-ample for any $\alpha\gg 0$.\\

\emph{Step 5.} In general, $D_{\overline{S}}$
may not be nef. However, the LMMP on $K_S+\Gamma_{S,\alpha}$  in Step 3 is a partial LMMP on $D_S$, that is, 
each step of the LMMP on $K_S+\Gamma_{S,\alpha}$ is also a step of an LMMP on $D_S$ but it may not be 
a full LMMP on $D_S$.
In any case, there is a $\Q$-divisor $\overline{P}\ge 0$ on $\overline{W}$ which is exceptional$/\overline{S}$ and
$$
c^*D_S=\overline{P}+d^*D_{\overline{S}}
$$
Now put $G=e^*d^*D_{\overline{S}}$ and $\phi:=\tau\mu$.
Then, for some integer $l>0$, we have natural isomorphisms 
$$
R(Y,lD_Y)\simeq R(S,lD_S)\simeq R(\overline{S},lD_{\overline{S}})\simeq R(W,lG)
$$

By construction,
$\alpha M_W+G$ is semi-ample for any $\alpha\gg 0$.
Since $M_W\sim_\Q 0/Y$,
we deduce that $G$ is semi-ample over $Y$.
Moreover,
\begin{equation*}
\begin{split}
 \phi^*D_Y+\mu^*R^- & \sim_\Q \mu^*h^*(K_S+\Delta_S+L_S+ bM_S)+\mu^*R^+ \\
  & \sim_\Q e^*c^*D_S+\mu^*R^+\\
  & = e^*\overline{P}+G+\mu^*R^+
\end{split}
\end{equation*}
Put $D_Y':=\phi_*(e^*\overline{P}+G+\mu^*R^+)$. Since $e^*\overline{P}+G+\mu^*R^+\ge 0$, we get 
$D_Y'\ge 0$, and since $R^-$ is exceptional$/Y$, we get $D_Y'\sim_\Q D_Y$. So,
$$
 \phi^*D_Y'+\mu^*R^-\sim_\Q e^*\overline{P}+G+\mu^*R^+
$$
and by the negativity lemma we have equality
$$
 \phi^*D_Y'+\mu^*R^-=e^*\overline{P}+G+\mu^*R^+
$$
Let $N=\phi^*D_Y'-G$. Then,
$$\phi_*N=\phi_*(N+\mu^*R^-)=\phi_*(e^*\overline{P}+\mu^*R^+)\ge 0
$$
On the other hand, we know that
$G$ is nef over $Y$ so $N$ is anti-nef over $Y$. Therefore, by the
negativity lemma we deduce $N\ge 0$ hence
$\phi^*D_Y'\ge G$. Now replace $D_Y$ with $D_Y'$ and this completes the proof of the proposition.
\end{proof}

\begin{proof}(of Theorem \ref{t-main})
This follows from Proposition \ref{p-aux} using the same notation.
\end{proof}

\section{Fujino-Gongyo conjecture}

As mentioned earlier the smoothness assumption of $f$ in Theorem \ref{t-FG-conj} cannot be 
removed. 
If $f:X\to Z$ is a surjective morphism between smooth projective varieties with $-K_X$ nef, then  Zhang \cite{zh96} proved that either $Z$ is uniruled or $\kappa(Z)=0$. In particular, if $\dim Z=1$, then $-K_Z$ is semi-ample. Therefore, a counter-example to Fujino-Gongyo's conjecture without the smoothness assumption on $f$ is possible only when $\dim Z\geq 2$. We give few counter-examples 
of different flavours.

\begin{exa}
Let $B'$ be  the $(4,4)$-divisor on $Z'=\mathbb{P}^1\times \mathbb{P}^1$ which consists of 4 lines vertical with respect to the 
first projection and 4 lines vertical with respect to the second projection. Let $\sigma \colon Z\rightarrow Z'$ be the blow up of the 16 intersection points in $B'$. Let $B$ be the strict transform of $B'$. Then $\sigma^*(B')=B+\sum_{i=1}^{16}2E_i$, where $E_i$ are the irreducible exceptional divisors over the 16 points. So $B$ is an even divisor in $Z$, i.e., $B\sim2L$ for some divisor $L$ in $Z$. Let $f \colon X\rightarrow Z$ be the double cover ramified over $B$. We will show that $-K_{X} $ is semi-ample but $-K_{Z}$ is not semi-ample.

By construction, $K^2_{Z'}=8$ and from $K_{Z}=\sigma^*K_{Z'}+\sum_{i=1}^{16}E_i$ we get $K^2_{Z}=-8$. Therefore, $-K_{Z}$ is not nef, and thus not semi-ample. Since $f:X\rightarrow Z$ is the double cover ramified over $B$, $K_{X}\sim_\Q f^*(K_{Z}+L)$.  On the other hand, 
$$
K_{Z}+L\sim_\Q \sigma^*K_{Z'}+\frac{1}{2}(\sigma^*(B')-B)+\frac{1}{2}B=\sigma^*(K_{Z'}+\frac{1}{2}B')\sim_\Q 0
$$
So $-K_{X}\sim_\Q 0$ is semi-ample. In this example $f$ is a flat morphism but not a contraction.
\end{exa}

\begin{exa}
 Let $r$ and $s$ be positive integers. Let $\x{E}=\x{O}_{\mathbb{P}^s}\oplus\x{O}_{\mathbb{P}^s}(1)^{r+1}$  and let $Z_{r,s}$ be the smooth $(r+s+1)$-dimensional variety $\mathbb{P}(\x{E})$. The projection $\pi_{r,s}\colon Z_{r,s}\rightarrow \mathbb{P}^s$ has a section $\sigma$ with image $P_{r,s}\subset Z_{r,s}$ corresponding to the trivial quotient $\x{E}\to \x{O}_{\mathbb{P}^s}$.  Let $f\colon X_{r,s}\rightarrow Z_{r,s}$ be the blow-up of $P_{r,s}$. One can check by calculating intersection numbers and using Kleiman's ampleness criterion that 
 when $r>s$ the divisor $-K_{X_{r,s}}$ is ample but $-K_{Z_{r,s}}$ is not nef. For details see \cite{de01} Example 1.36 and Example 3.16 (2). 
In this example $f$ is a contraction but not a flat morphism.
\end{exa}

For the convenience of the reader we reproduce 
another counter-example due to Hiroshi Sato given in \cite{fg11}, Example 4.6. In this case $f$ would be 
both a contraction and a flat morphism.

\begin{exa}[H. Sato]\label{exa-counter} 
Let $\Sigma$ be the fan in $\mathbb{R}^3$ whose rays are generated by
$$\begin{array}{l}x_1=(1,0,1), x_2=(0,1,0),x_3=(-1,3,0),x_4=(0,-1,0),\\
y_1=(0,0,1),y_2=(0,0,-1),\end{array}$$
and their maximal cones are
$$\begin{array}{c}\langle x_1,x_2,y_1\rangle,\langle x_1,x_2,y_2\rangle, \langle x_2,x_3,y_1\rangle,\langle x_2,x_3,y_2\rangle,\\
\langle x_3,x_4,y_1\rangle,\langle x_3,x_4,y_2\rangle, \langle x_4,x_1,y_1\rangle,\langle x_4,x_1,y_2\rangle.
\end{array}$$
Let $\Delta$ be the fan obtained from $\Sigma$ by successive star subdivisions along the rays spanned by $$z_1=x_2+y_1=(0,1,1)$$ and
$$z_2=x_2+z_1=2x_2+y_1=(0,2,1).$$
We can see that $V=X_\Sigma$, the toric threefold corresponding to the fan $\Sigma$ with respect to the lattice $\mathbb{Z}^3\subset \mathbb{R}^3$, is a $\mathbb{P}^1$-bundle over $Z=\mathbb{P}_{\mathbb{P}^1}(\x{O}_{\mathbb{P}^1}\oplus\x{O}_{\mathbb{P}^1}(3))$. The $\mathbb{P}^1$-bundle structure $V\rightarrow Z$ is induced by the projection $\mathbb{Z}^3\rightarrow \mathbb{Z}^2:(x,y,z)\mapsto (x,y)$. The toric variety $X=X_\Delta$ corresponding to the fan $\Delta$ was obtained by successive blow-ups from $V$. The maximal cones of $X_\Delta$ are:
$$\begin{array}{c} \tau_1=\langle x_1,x_2,z_2\rangle, \tau_2=\langle x_1,z_1,z_2\rangle, \tau_3=\langle x_1,z_1,y_1\rangle,\\
\tau_4=\langle x_3,y_1,z_1\rangle, \tau_5=\langle x_3,z_2,z_1\rangle, \tau_6=\langle x_3,z_2,x_2\rangle,\\
\sigma_3=\langle x_1,x_2,y_2\rangle,  \sigma_4=\langle x_2,x_3,y_2\rangle, \sigma_5=\langle x_3,x_4,y_1\rangle,\\
 \sigma_6=\langle x_3,x_4,y_2\rangle, \sigma_7=\langle x_4,x_1,y_1\rangle, \sigma_8=\langle x_4,x_1,y_2\rangle.
\end{array}$$

 To check  $-K_X$ is semi-ample, we can either apply  the base-point-free criterion on toric varieties (see \cite{cls} Theorem 6.1.7), or the fact that on toric varieties,  semi-ampleness of a  divisor is equivalent to nefness (see \cite{cls} Theorem 6.3.12). Here we apply the base-point-free criterion (see Remark \ref{rem-bpf-toric}).

We have $-K_X=\sum_{i=1}^4D_{x_i}+D_{y_1}+D_{y_2}+D_{z_1}+D_{z_2}$, where $D_{u}$ denotes the invariant Cartier divisor corresponding to the ray generated by the point $u$. The Cartier data associated to $-K_X$ are as follows:
$$\begin{array}{l}m_{\tau_1}=m_{\tau_6}=m_{\sigma_3}=m_{\sigma_4}=(-2,-1,1),\ \ \ \ \ m_{\tau_2}=m_{\tau_3}=(0,0,-1),\ \ \ \ \\\ m_{\tau_4}=m_{\tau_5}=(1,0,-1),\ \ m_{\sigma_5}=(4,1,-1),\ \ \ \ \ m_{\sigma_6}=(4,1,1),\ \ \\\ m_{\sigma_7}=(0,1,-1),\ \ \ \ \ m_{\sigma_8}=(-2,1,1).
\end{array}$$

One can check that all $m_{\tau_i},m_{\sigma_j}$ are in  $P_{-K_X}$. So $-K_X$ is base point free. Moreover, 
 $-K_X$ is  big as $X$ is a projective toric variety. So $X$ is a toric weak Fano manifold.  The morphism $f:X\rightarrow Z$ induced by the projection $\mathbb{Z}^3\rightarrow \mathbb{Z}^2$ is a flat morphism onto $Z$ since every fiber of $f$ is one-dimensional. 

On the other hand, $-K_Z$ is not nef since $-K_Z\cdot C<0$ where 
$C\subset Z$ is the image of the section of $Z\to \PP^1$ corresponding to the trivial quotient 
$\x{O}_{\mathbb{P}^1}\oplus\x{O}_{\mathbb{P}^1}(3)\to \x{O}_{\mathbb{P}^1}$.
\end{exa}

\begin{rem}\label{rem-bpf-toric}
Here we recall some basic facts about toric varieties. 

(1) Base point freeness on toric varieties: We follow the notation of \cite{cls}. Let $X_\Sigma$ be a complete toric variety of dimension $n$ and 
 let $D=\sum_{\rho}a_{\rho}D_{\rho}$ be a torus-invariant Cartier divisor on $X_{\Sigma}$. We define a polytope 
$$
P_D=\{m\in M_{\mathbb{R}}|\langle m,u_{\rho}\rangle\geq -a_{\rho}\text{ for all }\rho\in \Sigma(1)\}\subset M_{\mathbb{R}}
$$ 
Let $\{m_{\sigma}\}_{\sigma\in \Sigma(n)}$ be the Cartier data associated to $D$, where $m_{\sigma}\in M$ with $\langle m_{\sigma},u_{\rho}\rangle=-a_{\rho},\text{ for all }\rho\in \sigma(1)$. Then,  $D$ is base point free if and only if $m_{\sigma}\in P_D$ for all $\sigma\in \Sigma(n)$.

(2) Kodaira dimension: Now assume that $D$ is nef. Then, the Kodaira dimension $\kappa(D)=\dim P_D$. 
In particular, if $X_{\Sigma}$ is projective and $D=-K_{X_\Sigma}$ is nef, then it is not difficult 
to see that $\dim P_{D}=n$ hence $D$ is nef and big.\\
\end{rem}

\begin{proof}(of Theorem \ref{t-FG-conj})
By Lemma 2.4 of \cite{fg11}, we can assume that $f$ is a contraction: 
this is achieved by taking the Stein factorisation which preserves the smoothness assumption. 
We may assume that $\dim X>\dim Z$ otherwise the theorem holds trivially.
Pick a closed point $z\in Z$.
Since $-K_X$ is semi-ample, for some $m>1$ we can find a divisor $D\in |-mK_X|$ such that 
if we put $B=\frac{1}{m}D$ then we have: $K_X+B\sim_\Q 0$, 
$(X,B)$ is klt, and $D=\Supp B$ is  smooth (note that it is also possible 
to have $B=0$ depending on the situation). 
Since $f$ is smooth, we can choose $D$ so that there is a neighbourhood
$U$ of $z$ such that $D=\Supp B$ is relatively smooth over $U$.
In particular, this ensures that the discriminant part $B_Z=0$ near $z$.

Let $\sigma\colon Y\to Z$ be an Ambro model for $f$ and $(X,B)$. 
Then, possibly after shrinking $U$ around $z$, we have $M_Y=\sigma^*M_Z$ 
over $U$ by \cite{fg11}, Remark 4.3, or by \cite{ko07}, Proposition 8.4.9 (3) 
and Theorem 8.5.1.
 Since $K_X+B\sim_\Q 0$,
$$
K_Z+B_Z+M_Z\sim_\Q 0 ~~ \mbox{and} ~~ K_Y+B_Y+M_Y\sim_\Q 0
$$
 and
$$
K_Y+\Delta_Y+M_Y \sim_\Q D_Y:=E
$$
where $\Delta_Y$ and $E$ are as in the Introduction.
Since $E$ is exceptional$/Z$, $D_Y$ is the only effective $\Q$-divisor in its
$\Q$-linear equivalence class. Moreover, we have $\kappa(K_Y+\Delta_Y/T)\ge 0$ since 
$$
\kappa(K_Y+\Delta_Y+M_Y)=\kappa(D_Y)\ge 0
$$ 
where $Y\to T$ is again as in the 
Introduction.

By Theorem \ref{t-main}, there is a resolution
$\phi\colon W\to Y$ and a $\Q$-divisor $G\ge 0$ such that $\alpha M_{W}+G$ is semi-ample for any $\alpha \gg 0$
and $G\le \phi^*D_Y=\phi^*E$.
In particular,
this means that $\alpha M_{W}+G$ is semi-ample over $U$ for $\alpha \gg 0$.
Since over $U$ we have $M_W=\phi^*\sigma^*M_Z$, we deduce that $G$ is semi-ample over $U$.
But $G$ is effective and exceptional over $U$ hence by the negativity lemma it is zero over $U$, that is,
$G$ is mapped into $Z\setminus U$. Thus, we can find a rational number $\alpha\gg 0$ and a  
$\Q$-divisor $P_W\ge 0$ such that $\alpha M_{W}+G\sim_\Q P_W$ and that $P_W$ does not contain any 
irreducible component of $(\sigma\phi)^{-1}\{z\}$. But $P_W$ restricted to $(\sigma\phi)^{-1}\{z\}$ 
is numerically trivial so $P_W$ does not intersect $(\sigma\phi)^{-1}\{z\}$.
Therefore, 
$$
\alpha M_{Z}\sim_\Q P_Z:=\sigma_*\phi_*P_W\ge 0
$$ and $z$ does not belong to $\Supp P_Z$  which 
implies that the stable base locus of $M_Z$ does not
contain $z$. By construction, $\Supp B_Z$ also does not contain $z$. Therefore, $z$ is not 
in the stable base locus of $-K_Z$. Since $z$ was chosen arbitrarily, we deduce that 
$-K_Z$ is semi-ample.\\
\end{proof}


\vspace{2cm}

\flushleft{DPMMS}, Centre for Mathematical Sciences,\\
Cambridge University,\\
Wilberforce Road,\\
Cambridge, CB3 0WB,\\
UK\\
email: c.birkar@dpmms.cam.ac.uk

\vspace{1cm}

\flushleft{Institute} of Mathematics,\\
         Chinese Academy of Sciences,\\
         No. 55 Zhonguancun East Road,\\
         Haidian District, Beijing, 100190,\\
         China\\
email: yifeichen@amss.ac.cn

\end{document}